\documentclass{amsart}
\usepackage{amssymb,stmaryrd}
\usepackage{amsfonts}
\usepackage{amstext}
\usepackage{algorithmic}
\usepackage{algorithm}
\usepackage{graphicx}
\usepackage{epstopdf}
\usepackage[all]{xy}

\numberwithin{equation}{section}
\newtheorem{theorem}{Theorem}[section]
\newtheorem{lemma}[theorem]{Lemma}
\newtheorem{proposition}[theorem]{Proposition}
\newtheorem{corollary}[theorem]{Corollary}

\theoremstyle{definition}

\theoremstyle{remark}

\newcommand{\sch}{{\rm{Sch}_k}}

\newcommand{\C}{{\mathbb{C}}}
\newcommand{\Z}{{\mathbb{Z}}}

\newcommand{\F}{{\mathbb{F}}}
\newcommand{\N}{{\mathbb{N}}}
\newcommand{\A}{{\mathbb{A}}}

\newcommand{\Tr}{{\rm Trace}}
\newcommand{\CC}{{\mathcal{C}}}
\newcommand{\CS}{{\mathcal{S}}}
\newcommand{\CH}{{\mathcal{H}}}

\renewcommand{\ker}{{\rm{ker}}}
\newcommand{\tens}{\otimes}
\newcommand{\id}{\rm id}
\newcommand{\ev}{{\rm ev}}
\newcommand{\coev}{{\rm coev}}
\newcommand{\und}{\underline}

\begin{document}

\title{ON BRAIDED ZETA FUNCTIONS}

\author{Shahn Majid \& Ivan Toma\v si\'c}
\address{Queen Mary University of London\\
School of Mathematical Sciences, Mile End Rd, London E1 4NS, UK}
\thanks{The 1st author was supported by a Senior Leverhulme Research Fellowship}
\email{s.majid@qmul.ac.uk, i.tomasic@qmul.ac.uk}


\begin{abstract}
We propose a braided approach to zeta-functions in $q$-deformed geometry, defining $\zeta_t$ for any rigid object in a ribbon braided category. We compute $\zeta_t(\C^n)$ where $\C^n$ is viewed as the standard representation in the category of modules of $U_q(sl_n)$ and $q$ is generic.  We show that this coincides with $\zeta_t(\C^n)$ where $\C^n$ is the $n$-dimensional representation in the category of $U_q(sl_2)$ modules and that this equality of the two braided zeta functions is equivalent to the classical Cayley-Sylvester formula for the decomposition into irreducibles of the symmetric tensor products $S^j(V)$ for $V$ an irreducible representation of $sl_2$. We obtain functional equations for the associated generating function. We also discuss $\zeta_t(C_q[S^2])$ for the standard $q$-deformed sphere.  \end{abstract}
\maketitle 

\section{Introduction} 

The analogue of the Riemann hypothesis in the case of the zeta function of an algebraic variety over a finite field was proven by Deligne in \cite{Deligne-Weil1} and \cite{Deligne-Weil2}, as a culmination of a body of work developed by the Grothendieck school. However, in quantum group theory there are suggested links and analogies between algebraic geometry over finite fields and $q$-deformations in characteristic zero both generically and at roots of unity. There are also `reduced' finite-dimensional but noncommutative versions of algebraic varieties which capture some of the classical geometry much as working over $\F_q$. This suggests that a $q$-deformed Riemann hypothesis based on $q$-deformed geometry should be possible to formulate and prove. The present work is intended as a step in this direction. Other key ingredients such as $q$-complex geometry will be considered elsewhere, following the model of $q$-deformed $\C P^1$ in \cite{Ma:sph}. Such a quantum groups approach should not be confused with a rather more sophisticated approach of Connes using operator algebra methods~\cite{Con}.

Let $\sch$ be the category of algebraic schemes over a fixed ground field $k$ and let $S$
be a commutative ring with identity. The motivic $\zeta$-function is defined with reference to a  `weak motivic measure' on $\sch$, meaning a map $\llbracket\cdot\rrbracket:\sch\to S$ such that \cite{krajicek-scanlon}: 

\begin{enumerate}
\item\label{norm} for a finite scheme $X$, $\llbracket X\rrbracket =|X|$;
\item\label{isom} if $X\simeq Y$, then $\llbracket X\rrbracket=\llbracket Y\rrbracket$;
\item\label{sum} $\llbracket X\coprod Y\rrbracket =\llbracket X\rrbracket +\llbracket Y\rrbracket $;
\item\label{prod} $\llbracket X\times Y\rrbracket=\llbracket X\rrbracket\llbracket Y\rrbracket$.
\end{enumerate}
In the literature one also considers stronger notions,  where condition (3) 
is replaced by an `excision' axiom and the last requirement can be strengthened by requiring that if  $f:X\to Y$ is a map such that for every $y\in Y$, 
$\llbracket f^{-1}(y)\rrbracket=\mu$, then $\llbracket X\rrbracket =\mu\llbracket Y\rrbracket $. However, for our purposes the listed features are sufficient.

Given a fixed  weak motivic measure $\llbracket\cdot\rrbracket$ on $\sch$, for every scheme $X$ one defines (following Kapranov \cite{kapranov}): \[ \zeta_t(X)=\sum_j \llbracket X^{(j)}\rrbracket t^j\ \ \ \in S[[t]],\]
where $X^{(j)}$ is the $j$-th symmetric power of $X$, i.e., $X^{(j)}=X^j/S_j$ and the action of $S_j$ is permutation of the factors of the direct product. 

If we specialise to the case where $k=\F_q$ is a finite field, and $\llbracket X\rrbracket=|X(\F_q)|$
is the number of points over the ground field of a given variety, we recover the classical Hasse-Weil zeta function. However, the motivic definition is the one which we shall consider as this is rather more general. We also recall for orientation that for $\A^n$, $\Bbb P^n$ respectively the  affine and projective spaces of dimension $n$ and $k=\F_q$, 
\[ \zeta_t(\A^n)=\left({1\over 1-qt}\right)^n,\quad \zeta_t(\Bbb P^n)=\prod_{i=0}^n{1\over 1-q^i t}.\]

Now let us consider how the motivic version might be generalised for noncommutative geometry.  Instead of a variety $X$ we might work with an algebra $A$, possibly noncommutative but regarded as if a coordinate ring. We should have a notion of `finite scheme' which may be a bit more than merely finite-dimensional. Then a `weak motivic measure' would be a map $\llbracket\cdot\rrbracket$ from a reasonable class of algebras to a (commutative) ring $S$ such that
\begin{enumerate}
\item\label{Anorm} if $A$ `finite', $\llbracket A\rrbracket =\dim(A)$;
\item\label{Aisom} if $A\simeq B$, then $\llbracket A\rrbracket=\llbracket B\rrbracket$;
\item\label{Asum} $\llbracket A\oplus B\rrbracket = \llbracket A\rrbracket +\llbracket B\rrbracket $;
\item\label{Aprod} $\llbracket A\tens B\rrbracket=\llbracket A\rrbracket\llbracket B\rrbracket$.
\end{enumerate}
In place of $X^j$ we could consider $A^{\tens j}$ as an algebra and in place of $X^{(j)}$ we could define $A^{(j)}=(A^{\tens j})^{S_j}\subset A^{\tens j}$  as invariant subalgebra under $S_j$. Here $S_j$ acts by permutation of tensor factors of $A^{\tens j}$ and one can check that this action respects the product of $A^{\tens j}$ so that this makes sense, even when $A$ is noncommutative. Then
 \begin{equation}\label{zeta} \zeta_t(A)=\sum_{j=0}^\infty t^j \llbracket A^{(j)}\rrbracket.\end{equation}
For a finite-dimensional algebra we could consider $\llbracket A\rrbracket =\dim(A)$ as the dimension of the left-regular representation, but in this case $\zeta_t(A)=(1/(1-t))^{\dim(A)}$ independently of the algebra structure of $A$. The issue then is how to define $\llbracket\cdot\rrbracket$ more generally. In noncommutative geometry the $\tens$ of algebras would also need to be an appropriate one and might be more general than the usual tensor product or completion thereof.

Our proposal in this paper is to go down a different route to `counting points', namely to use a version of the categorical rank or `braided dimension' $\und\dim$ of an object in a braided category to measure its `size'. Here $\und\dim(C_q[G])$  was used in  \cite{MaSoi2} to  `count' the points in the Drinfeld-Jimbo-type quantum group coordinate algebras $C_q[G]$, with results related to $L$-functions in number theory (more precisely, this was done in terms of the quantum enveloping algebras $U_q(g)$). The notion $\und\dim$, however, is not multiplicative and we propose, in the case of a ribbon braided category, to use a variant $\und\dim'(A)$ which is. Thus we take $\llbracket A\rrbracket=\und\dim'(A)$ on a reasonable class of objects. This may include a sum of finite-dimensional (rigid) objects in which case we extend $\und\dim'$ additively and, in $q$-deformed examples with $q$ generic will obtain a powerseries in $q$ (typically a rational function singular at $q=1$). We also have in mind a setting where $q$ is a root of unity and the coordinate algebra $A$ can be finite-dimensional. Computations in the root of unity setting are deferred to a sequel while in this note $q$ is  always taken to be generic. More details are in Section~3.

Next we define  $A^{(j)}$ in a way that  makes sense in the ribbon braided category for an object $A$. Remembering our motivation, we define $A^{(j)}$ essentially as the invariant subobject of $A^{\tens j}$ under all adjacent braidings. More precisely, for $q$-deformation quantum groups we replace the role of the symmetric group $S_j$ by  a $q$-Hecke algebra or its appropriate variant.  We then define $\zeta_t(A)$ as before by (\ref{zeta}).  Finally, if 
the category has direct sums and if we have the $\lambda$-ring property
\begin{equation}\label{lring} (A\oplus B)^{(j)}\cong \sum_{k,l\ge 0; k+l=j} A^{(k)}\tens B^{(l)}\end{equation}
as objects then multiplicativity of the braided dimension $\und\dim'$ implies that
\[ \zeta_t(A\oplus B)=\zeta_t(A)\zeta_t(B).\]
This holds in our $q$-deformed examples for generic $q$ because the tensor products and direct sums have the same form as classically. 

It is still the case that our constructions are for an object  $A$ in a ribbon category without requiring an algebra structure, but if both structures exist they will be connected through functoriality. Our point of view is that the usual $\zeta$-function has a dependence that can be viewed as factoring through certain geometrical constructions on the variety and it is envisaged that some of that is now reflected in the $q$-deformed case  as an object in a ribbon braided category, even if the variety itself (expressed in an algebra structure on $A$) is not being directly referenced.  It is also possible to interpret the braided $\zeta$-function in the $q$-deformed case for generic $q$ as a braided Hilbert series. We will explain this in Section~5.

In this note we show that our approach is `reasonable' by computing $\zeta_t(\C^n)$ where $\C^n$ will be regarded in two different ribbon braided categories. On the one hand we look at the braided finite geometry of `$n$ points' in the sense $A=\C^n$ equipped with a $U_q(sl_n)$-induced solution of the braid relations, and on the other hand the same but equipped with a $U_q(sl_2)$-solution. The relevant braiding matrices are different yet we find, remarkably, the same resulting $\zeta_t$. Our result is, moreover, similar to the classical $\zeta$-function for $\Bbb P^n$ cited above.  

We view the remarkable equality here is indicating that the $\zeta$-function is invariant under a certain `coherence of $q$-deformation'  between different quantum groups in the same family. In our case it is implied by (and can be used to deduce) the classical Cayley-Sylvester formula  for the classical decomposition into irreducibles of the symmetric tensor powers $S^j(V_m)$ for all $j$, where $V_m$ is the $m+1$-dimensional irreducible representation of $sl_2$. Here our $\zeta_t$ obey functional equations which we interpret as functional equations for the the $q$-generating functions
$c_m(t,q)$ of the multiplicities of the irreducibles in all  $S^j(V_m)$. This provides a new 
way of computing these multiplicities which may be of interest in classical
representation theory.

For completeness, we discuss the definition by similar methods of $\zeta_t(C_q[S^2])$, i.e. for the quantum sphere coordinate algebra. The issue here is the ordering of infinite sums and we compute  one such ordering to degree $t^3$. 

\subsection*{Acknowledgements} We thank Malek Abdesselam for pointing us to the classical Cayley-Sylvester formula used in Section~4. We also thank Thomas Prellberg for discussions and for computing the explicit formula after Lemma~4.6. 
 
\section{Classical $\zeta$-functions for finite sets}

We start with some elementary remarks in the finite case. Let $X$ be a finite set.  The number of orbits of $S_j$ acting on $X\times \cdots\times X$ is the number of partitions of $j$ into $|X|$ parts because for each element of $X$ we have some non-negative integer for the number of times it occurs and a total of $j$ occurrences. This number of partitions is  $\left({j+|X|-1\atop |X|-1}\right)$, which gives
\begin{equation}\label{zetaset} \zeta_t(X)=({1\over 1-t})^{|X|}.\end{equation}
We see that the answer has the same as for an affine space of dimension $n=|X|$ and $q=1$ (in keeping with a certain philosophy about geometry over $\F_1$).

Note that  $X^j$ has  points with nontrivial stabiliser under the action of $S_j$. If we ignore these points then the `regular points' of $X^{(j)}$ are the orbits of $j$-tuples with distinct entries. These are $\left({|X|\atop j}\right)$ in number. Hence
\[ \zeta^{\rm regular}_t(X)=(1+t)^{|X|}= [2]_t^{|X|}\]
where we use the notation $[m]_t=(1-t^m)/(1-t)=1+t+\cdots +t^{m-1}$.  This compares with the full number of orbits including degenerate ones above which in this notation is $[\infty]_t^{|X|}$. 

Now let us look at this from an algebraic point of view over a field $k$. Then $n$ points could be viewed as corresponding to $A=k^n$ (or any other $n$-dimensional algebra).  We are interested only in the dimension of the vector space
\[ (A^{\tens j})^{S_j}\]
of symmetric tensors. At least if $k$ has characteristic zero, this space  has dimension $\left({\dim(A)+j-1\atop j}\right)$ because it can be identified with the degree $j$ part of the symmetric algebra $S^j(A)$. Thus, at least in this case
\[ \zeta_t(A)=({1\over 1-t})^{\dim(A)}\]
 can be viewed as  the Hilbert series of $S(A)$ and is in accord with (\ref{zetaset}). This serves as a model for what follows.

\section{Braided dimension} 

Let $\CC$ be a braided category and let $V$ be a rigid object, i.e. with (say) left dual. This means morphisms $\ev_V: V^*\tens V\to \und 1$ and $\coev_V:\und 1\to V\tens V^*$ obeying some usual axioms. Here $\und 1$ is the identity object with respect to the $\tens$ product. Typically the category will be $k$-linear and $\und 1=k$ and in this case if $\{e_a\}$ is a basis of $V$ and $\{f^a\}$ a dual basis then we may take $\coev_V(\lambda)=\lambda\sum_a e_a\tens f^a$, while $\ev_V$ is the usual evaluation. The braiding means that for any two objects $V,W$  there is an isomorphism $\Psi_{V,W}:V\tens W\to W\tens V$ obeying the usual coherence and functoriality properties (as a natural transformation $\tens \to \tens^{\rm op}$). We refer to \cite[Chapter 9]{Ma:book} for an introduction. In this case there is a natural `rank' or `braided dimension'
\[ \und{\dim}(V)=\ev_V\circ\Psi_{V,V^*}\circ\coev_V\]
as a map $k\to k$, which we view as an element of $k$ acting by multiplication. There is an associator for the bracketting of  tensor products which is omitted here but should be understood (Mac Lane's coherence theorem says that we can insert the associator as needed for compositions to make sense and different ways to do this will give the same net result). One could use this braided dimension but it is not multiplicative when the category is strictly braided.

Next we recall that a ribbon structure on a rigid braided category is a natural isomorphism $\nu_V:V\to V$ (a coherent collection of isomorphisms) such that
\[ \nu_{V\tens W}=\Psi^{-1}_{V,W}\circ\Psi^{-1}_{W,V}\circ(\nu_V\tens \nu_W), \quad\nu_{\und 1}=\id,\quad\nu_{V^*}=(\nu_V)^*.\]
This allows the formation of right duals as well as the left ones assumed, which is relevant to applications in knot theory, but also allows the formation of $\und\dim'$ defined as a braided trace $\und\Tr$ of $\nu$ in \cite{Ma:book}
\[ \und\dim'(V)=\und\Tr(\nu_V)=\ev_V\circ\Psi_{V,V^*}\circ(\nu_V\tens\id)\circ\coev_V.\]
It is shown in \cite{Ma:book} that this is multiplicative for $\tens$. When the category has direct sums we again extend $\und\dim'$ to sums of rigid objects and we define 
\[ \llbracket A\rrbracket=\und\dim'(A).\]

For the standard quantum groups $U_q(g)$ associated to complex simple Lie algebras \cite{Dri} the category of finite-dimensional highest weight modules is a ribbon rigid braided category with direct sums. There are different formulations but we will work with generic $q$ over $k=\C$ and everything $k$-linear. For generic $q$ the objects, their tensor products and their decompositions into irreducibles follow the same pattern up to isomorphism as in the classical case and we will make use of this.  When we consider infinite sums we will be working in a slightly larger category.  For orientation purposes, for generic $q$ the standard $j+1$-dimensional representations $V_j$ of the quantum group $U_q(sl_2)$ are labelled as for the corresponding $sl_2$-representations by non-negative integers $j$. In these conventions,
\[ \und\dim(V_j)=q^{-{j(j+2)\over 2}}\und\dim'(V_j);\quad  \und\dim'(V_j)={q^{j+1}-q^{-(j+1)}\over q-q^{-1}}= (\dim(V_j))_q\]
where $(n)_q=(q^n-q^{-n})/(q-q^{-1})$ is the `symmetric $q$-integer' version of an integer $n$.

As a concrete application of these ideas to `counting points', consider  $C_q[S^2]$ defined as a $U(1)$-invariant subalgebra of $C_q[SU_2]$ under right translation (in the sense of a $U_q(su_2)$-module restricted to the maximal torus), see \cite{Ma:sph,Ma:book} for an exposition. The notation comes from the compact real form, but we will not be concerned about this explicitly and will work at the level of  $U_q(sl_2)$-modules. It is also known for generic $q$ that there is a Peter-Weyl decomposition $C_q[SU_2]=\oplus_{j\ge0} V_j\tens V_j^*$ with left and right translations on $V_j,V_j^*$ respectively, following the same form as classically. From this it follows that $C_q[S^2]=\oplus_{j\ge0,j\in 2\Z} V_j$ since only $V_j^*$ for even $j$  have an invariant element under the maximal torus. From this one may readily compute that 
\[ \und\dim(C_q[S^2])={1\over 1-q^{-2}}, \quad  \und\dim'(C_q[S^2])={2\over (1-q^{-2})(1-q^2)}.\]
Other $q$-deformation quantum groups and their homogeneous spaces can be treated similarly.  Although we are not aware of a formal treatment of this issue, one typically has reasonable functions $q$ in the numerator (such as $L$-functions \cite{MaSoi2}) and a pole at $q=1$. 

Now let $A$ be an object in a $k$-linear  braided category. In this case we have an action of the braid group $B_j$ on $A^{\tens j}$ via the braiding $\Psi$ between adjacent copies. The composition of these $\Psi$ generate a subalgebra $H_j$ of the endomorphisms of $A^{\tens j}$. This subalgebra has an action on $A^{\tens j}$  and we define
\[ A^{(j)}=(A^{\tens j})^{H_j}.\]
using an appropriate normalisation of the action of the generators in the strictly braided case. When $A$ is an algebra and the category symmetric  we obtain a subalgebra of the braided tensor product $A^{\und\tens j}$, but in general we will not.  When our braided category is ribbon  we define
\begin{equation}\label{braidedzeta} \zeta_t(A)=\sum_{j=0}^\infty t^j \und\dim'(A^{(j)}).\end{equation}

 \section{Braided zeta function of $\C^n$}

This section contains the main result of the paper, a computation of $\zeta_t(\C^n)$  in the ribbon braided category of finite-dimensional $U_q(sl_n)$-modules, i.e. in some sense for $n$ `braided points'. We  also find that this coincides with $\zeta_t(\C^n)$ in the ribbon category of finite-dimensional $U_q(sl_2)$-modules, where we regard $\C^n=V_{n-1}$ and we relate this to classical formulae in the representation theory for $sl_2$.

For $n=1$ the quantum group is trivial as are all $(\C^{\tens j})^{S_j}=\C$; clearly $\zeta_t(\C)=1/(1-t)$ as for a single point. 

For $n=2$ the classical picture is $((\C^2)^{\tens j})^{S_j}\cong S^j(\C^2)$ which has dimension $j+1$ and is given by the polynomials in $x,y$ of degree $j$. These are known to form the irreducible $j+1$-dimensional representation $V_j$. It is also known that dimensions and multiplicities do not change under deformation for generic $q$; $S_j$ becomes the standard $q$-Hecke algebra $\CH_{q,j}$ and $((\C^2)^{\tens j})^{\CH_{q,j}}$ is a $U_q(sl_2)$-module of the same dimension as classically, i.e. $V_{j}$ now as a $U_q(sl_2)$-module. Up to a normalisation the action of the abstract $q$-Hecke generators is by the braiding of adjacent copies as in Section~3. Hence the braided zeta function in the category of $U_q(sl_2)$-modules is
\[ \zeta_t(\C^2)=\sum_{j=0}^\infty t^j{q^{j+1}-q^{-(j+1)}\over q-q^{-1}}={1\over (1-qt)(1-q^{-1}t)}.\]

\begin{proposition}\label{zetacn} Regarding $\C^n$, $n\ge 2$ as a fundamental  $U_q(sl_n)$-module, for generic $q$, we have
\[\zeta_t(\C^n)=\prod_{j=-(n-1)\\ {\rm step} 2}^{n-1}{1\over 1-q^jt}.\]
\end{proposition}
\begin{proof} The classical picture is that  $((\C^n)^{\tens j})^{S_j}\cong S^j(\C^n)$ is an irreducible representation of $sl_n$ with highest weight $j\omega_1$. This corresponds to a Young diagram with just one row, of $j$ boxes. The corresponding Young symmetrizer is the total symmetrization in $j$ tensor powers, i.e. the image in $(\C^n)^{\tens j}$ is our desired space of invariants. In the $q$-deformed case we will define braided-symmetrization as  $((\C^n)^{\tens j})^{\CH_{q,j}}$, the space of invariants under the appropriate $q$-Hecke algebra and for generic $q$ this will be the corresponding irreducible representation of $U_q(sl_n)$. Now for all the standard quantum groups associated to complex simple Lie algebras $g$ the usual braided dimension and the multiplicative one of the highest weight representation $V(\Lambda)$ are
\[ \und\dim(V(\Lambda))=q^{-(\Lambda,\Lambda+2\rho)}\und\dim'(V(\Lambda));\quad \und\dim'(V(\Lambda))=\prod_{\alpha>0}{\left((\alpha,\Lambda+\rho)\right)_q\over \left((\alpha,\rho)\right)_q}\]
see \cite{MaSoi2}. We use the {\em symmetric} $q$-integers $(m)_q=(q^m-q^{-m})/(q-q^{-1})$ as in Section~4. Here $\alpha$ are positive roots and $\rho={1\over 2}\sum_{\alpha>0}\alpha=\sum_{i=1}^l\omega_i$ where $l={\rm rank}(g)$ and $\{\omega_i\}$ are the basis of the weight lattice defined by $(\omega_i,\check{\alpha}_j)=\delta_{i,j}$. Note that $(\alpha,\rho)$ is the height or the number of simple roots making up $\alpha$. Here $\check\alpha_i=2\alpha_i/(\alpha_i,\alpha_i)$ are the coroots. We compute this for $sl_n$ for our particular $\Lambda=j\omega_1$. 

Note that the standard normalisation for $sl_n$, which we use, is with the simple roots all with $(\alpha_i,\alpha_i)=2$. Then $\check{\alpha}_i=\alpha_i$. We also chose a standard ordering of the positive roots which begins $\alpha_1,\alpha_1+\alpha_2,....,\alpha_1+\alpha_2+\cdots+\alpha_{n-1}$. Then similarly $\alpha_2,\alpha_2+\alpha_3,\cdots,\alpha_2+\cdots+\alpha_{n-1}$, and so on. Now, in the product we will only have factors from $\alpha$ containing $\alpha_1$, i.e. from the first $n-1$ in this ordering:
\[\begin{split}\und\dim'(V(j\omega_1))&=\prod_{\alpha>0}{\left((\alpha,j\omega_1+\rho)\right)_q\over \left((\alpha,\rho)\right)_q}=\prod_{i=1}^{n-1}{\left((\alpha_1+\cdots+\alpha_i,j\omega_1+\rho)\right)_q\over \left((\alpha_1+\cdots+\alpha_i,\rho)\right)_q}\\
&=\prod_{i=1}^{n-1}{q^{j+i}-q^{-(j+i)}\over q^i-q^{-i}}=\left({n+j-1\atop j}\right)_q\end{split}\]
where we use $q$-binomial coefficients defined by the symmetric $q$-integers. Hence
\[ \zeta_t(\C^n)=\sum_{j=0}^\infty t^j \left({n+j-1\atop j}\right)_q\]
as a $q$-deformation of the standard Hilbert-series of a vector space of dimension $n$. We now have to compute this, which we do by induction on $n$ with the $n=2$ case proven already (one could also start with $n=1$ with care). Suppose the result for $n-1$ and write $\zeta_t(\C^m)\equiv \zeta_m(t)$. Then
\[ \begin{split}\zeta_t(\C^n)&=\sum_{j=0}^\infty t^j\left({q^{j+n-1}-q^{-(j+n-1)}\over q^{n-1}-q^{-(n-1)}}\right)\prod_{i=1}^{n-2}{q^{j+i}-q^{-(j+i)}\over q^i-q^{-i}}\\
&={q^{n-1}\zeta_{n-1}(qt)-q^{-(n-1)}\zeta_{n-1}(q^{-1}t)\over q^{n-1}-q^{-(n-1)}}  \\
&= { q^{n-1}\prod^{n-2}_{j=-(n-2) {\rm step} 2}{1\over 1-q^{j+1}t}-q^{-(n-1)}\prod_{j=-(n-2) {\rm step} 2}^{n-2}{1\over 1-q^{j-1}t}\over q^{n-1}-q^{-(n-1)}} \\
&={\left({q^{n-1}\over 1-q^{n-1}t}-{q^{-(n-1)}\over 1-q^{-(n-1)}t}\right)\over q^{n-1}-q^{-(n-1)}}\prod_{j=-(n-3) {\rm step} 2}^{n-3}{1\over 1-q^jt}\\
&={1\over (1-q^{n-1}t)(1-q^{-(n-1)}t)}\prod_{j=-(n-3) {\rm step} 2}^{n-3}{1\over 1-q^jt}\end{split}\]
as required.  This can be viewed as a symmetric $q$-integer version of the Newton formula for $q$-binomials and is presumably known to experts on $q$-series. \end{proof}

It would be interesting (but significantly harder) to compute the result for $q$ a root of unity. We also note the apparent similarity with the weight structure of the $n$-dimensional representation of $sl_2$ under the maximal torus.  This and the general philosophy that deformation in the $A$-series involves replacing integers by $q$-integers, suggests the following theorem. It expresses a kind of invariance of $\zeta_t$  within the same quantum group family.

\begin{theorem}\label{conj} For generic $q$ the braided $\zeta_t(\C^n)$ for $\C^n$, $n\ge 1$, regarded as the $n$-dimensional irreducible  $U_q(sl_2)$-module $V_{n-1}$, is given by the same formula as in Proposition~\ref{zetacn}\end{theorem}

This time the classical picture is quite different. We let $n=m+1$ from now on, so $\C^n=V_m$ as an $sl_2$-module. This time  $S^j(V_m)$ will in general not be an irreducible representation of $sl_2$. Rather, let
\[ S^j(V_m)=\oplus_{p\ge 0} c_m{}^j{}_p V_p,\quad c_m{}^j{}_p\in\N\cup\{0\}.\]
It is known that \cite{FulHar}
\[ S^j(V_0)=V_0,\quad S^j(V_1)=V_j,\quad S^j(V_2)=V_{2j}\oplus V_{2j-4}\oplus\cdots\oplus V_{2j-4\lfloor j/2\rfloor}\]
where $\lfloor\ \rfloor$ denotes integer part.  The general case $m>2$ is more complicated but there is a classical formula of Cayley-Sylvester 
\begin{equation}\label{CS}
S^j(V_m)=\oplus_{r=0}^{\lfloor{jm\over 2}\rfloor}V_{jm-2r}^{\oplus(p(r,j,m)-p(r-1,j,m))}\end{equation}
in terms of the number $p(r,j,m)$ of partitions of $r$ into at most $j$ parts each $\le m$, see for example \cite{Spr}. We shall use this formula to prove the theorem. Conversely, invariance of the braided $\zeta$-function expressed in the theorem provides a new point of view on this classical result and one which we expect to apply more widely.

Our first step is to introduce generating functions. We define
\[ c_m(t,q)=\sum_{j,p} c_m{}^j{}_p\, t^j q^p=\sum_p c_m(t)_pq^p.\]
So, the low cases appear now as
\[ c_0(t,q)={1\over 1-t},\quad c_1(t,q)={1\over 1-qt},\quad c_2(t,q)= {1\over (1-t^2)(1-q^2t)}.\]
The following lemma relates these to the braided $\zeta$-functions.
\begin{lemma} 
\[ \zeta_t(V_m)={q c_{m}(t,q)- q^{-1} c_{m}(t,q^{-1})\over q-q^{-1}}.\]
\end{lemma}
\proof From the above and Section~3 we have immediately 
 \[ \zeta_t(V_m)=\sum_j t^j \und\dim'(S^j(V_m))=\sum_{j,p} t^j c_m{}^j{}_p \und\dim'(V_p)=\sum_{j,p} t^j c_m{}^j{}_p{q^{p+1}-q^{-(p+1)}\over q-q^{-1}}\]
from which the stated formula follows.  \endproof 

\proof (Of Theorem~4.2). The partition number $p(r,j,m)$ occurs in the theory of the cohomology of Grassmannians $Gr_m(j+m)$ and in that context it is known that it can be viewed as the coefficient of $q^r$ in an asymmetric $q$-binomial coefficient defined by $[n]_q=(1-q^n)/(1-q)$. We recast this as the coefficient of $q^{-2r}$ in a binomial defined by $[n]_{q^{-2}}$ and hence the coefficient of $q^{jm-2r}$ in the symmetric $q$-binomial $\left({j+m\atop m}\right)_q$. The powers in this range from $q^{jm}$ to $q^{-jm}$ in steps of 2. In view of these observations, we can write
\[ c_m(t,q)=\sum_{j=0}^\infty t^j\sum_{r=0}^{\lfloor {jm\over 2}\rfloor}q^{jm-2r}(p(r,j,m)-p(r-1,j,m))\]
\[\quad \quad=(1-q^{-2})\sum_{j=0}^\infty t^j\left({j+m\atop m}\right)^+_q+q^{-2} \sum_{j=0}^\infty t^j \left({j+m\atop m}\right)^0_q\]
where $(\ )^+$ denotes the part with non-negative powers of $q$ while $(\ )^0$ denotes the term with lowest power (i.e. $q^0$ when $jm$ is even and $q^1$ when $jm$ is odd). Doing the sums as in the proof of Proposition~4.1 we obtain
\[ c_m(t,q)=(1-q^{-2})\zeta_{t,m}(q)^++q^{-2}\zeta_{t,m}(q)^0\]
where $\zeta_{t,m}(q)$ is the expression in Proposition~4.1 for $n=m+1$. We then use the lemma to find
\[ \zeta_t(V_m)={(q-q^{-1})\zeta_{t,m}(q)^++q^{-1}\zeta_{t,m}(q)^0-(q^{-1}-q)\zeta_{t,m}(q)^--q\zeta_{t,m}(q)^{-0}\over q-q^{-1}}=\zeta_{t,m}(q)\]
where we used that $\zeta_{t,m}$ is symmetric in $q,q^{-1}$. Here $(\ )^-$ denotes the part with non-positive powers of $q$ and $(\ )^{-0}$ denotes the term with $q^0$ or $q^{-1}$ as $jm$ is even or odd. In the even case the $\zeta_{t,m}(q)^0$ and $\zeta_{t,m}(q)^{-0}$ are equal and cancel the duplication in degree 0. In the odd case they cancel with each other. \endproof

The expression found for $c_m(t,q)$ in the course of the proof, while equivalent to the Cayley-Sylvester formula, is not very illuminating. We now turn to other ways to obtain $c_m(t,q)$ as a consequence of Theorem~4.2. First, in the same spirit, we can also recover $c_m(t,q)$ from
\[ (q-q^{-1})\zeta_t(V_m)= q c_m(t,q) - q^{-1}c_m(t,q^{-1})\]
since the first term on the right has only strictly positive powers of $q$ while the 2nd term has only strictly negative powers. For example, if $m=2s$  then, $(1-t)q c_m(t,q)$ is the truncation to strictly positive powers of $q$ in the expression
\[ (q-q^{-1})\prod_{i=1}^s{1\over (1-q^it)(1-q^{-i}t)}=(q-q^{-1})\sum_{\{k_i\ge 0\},\{j_i\ge 0\}} q^{2\alpha(k,j)}t^{\sum_i (k_i+j_i)}\]
where $\alpha(k_1,\cdots,k_s,j_1,\cdots, j_s)=\sum_{i=1}^s i(k_i-j_i)$. Hence
\[ (1-t)qc_m(t,q)=q\sum_{k,j; \alpha=0}t^{\sum (k+j)}+ (q-q^{-1})\sum_{k,j;\alpha\ge 1}q^{2\alpha}t^{\sum(k+j)}.\]
We can also deduce an inductive formula as follows.

\begin{corollary}\label{cmfromzeta}
\begin{eqnarray*} c_m(t,q)&=& {c_{m-2}(t,q)\over (1-q^mt)(1-q^{-m}t)} -{1\over 1-t^2}({q^{-m}t\over 1-q^{-m}t}\sum_p c_{m-2}(t)_pq^p(q^{-m}t)^{\lfloor p/m\rfloor} \\
&&\kern130pt+{q^{-2}\over 1-q^m t}\sum_p c_{m-2}(t)_p q^{-p}(q^m t)^{\lceil(p+2)/m\rceil})\end{eqnarray*}
where  $\lceil\ \rceil$ denotes the smallest integer greater than or equal to the enclosed quantity.
\end{corollary}
\proof From Theorem~4.2 and the form of $\zeta_t$ in Proposition~\ref{zetacn} we have
\[ (q-q^{-1})\zeta_t(V_m)={(q-q^{-1})\zeta_t(V_{m-2})\over(1-q^m t)(1-q^{-m}t)}={q c_{m-2}(t,q)-q^{-1}c_{m-2}(t,q^{-1})\over (1-q^m t)(1-q^{-m}t)}.\]
We expand the right hand side denominators and pick out the part which has strictly positive powers of $q$, thus
\begin{eqnarray*} qc_m(t,q)&=&\sum_{i,j,p\ge 0, (i-j)m+p+1>0}c_{m-2}(t)_p t^{i+j}q^{(i-j)m+p+1}\\
&&-\sum_{i,j,p\ge 0,(i-j)m-p-1>0}c_{m-2}(t)_p t^{i+j}q^{(i-j)m-p-1}
\end{eqnarray*}
There is no $q^0$ contribution as the two terms in the numerator cancel in this degree. The first sum is constrained by $(j-i)\le p/m$ and we split this into two parts, one where $i$ is replaced in favour of $k=i-j\ge 0$ and the other where $j$ is replaced by $k=j-i$ is in the range $0< k\le p/m$. The second sum is also a geometric series, for $k=j-i\ge (p+2)/m$.  Thus
\begin{eqnarray*} qc_m(t,q)&=&{qc_m(t,q)\over(1-t^2)(1-q^mt)}+\sum_{p,i\ge 0}\sum_{0<k\le {p\over m}}c_{m-2}(t)_p t^{2i+k}q^{-km+p+1}\\
&&-\sum_{p,i\ge 0, k\ge{p+2\over m}}c_{m-2}(t)_p t^{2i+k}q^{km-p-1}\end{eqnarray*}
leading to the formula stated on doing the $i$ and $k$ sums. Note that all denominators of rational functions in $t$ should be understood as power series in $t$. 
 \endproof

It  is a useful and nontrivial check, which we leave to the reader,  that applying this inductive formula to $c_0$ gives $c_2$.

\begin{corollary} \label{c3c4}
\[ c_3(t,q)={1-qt+q^2 t^2\over (1-t^4)(1-qt)(1-q^3 t)},\quad c_4(t,q)={1-q^2t+q^4t^2 \over(1-t^2)(1-t^3)(1-q^2t)(1-q^4t)}.\]
\end{corollary}
\proof  Corollary~4.4 applied to $c_1(t,q)$ in the form $c_1(t)_p=t^p$ provides $c_3(t,q)$, an exercise which we leave to the reader.  We similarly find $c_4(t,q)$ by Corollary~4.4 from $c_2(t)_{p}=t^{p/2}/(1-t^2)$ for even $p$ and zero for odd $p$. \endproof

We conclude with some remarks about the general form of $c_m(t,q)$. Let
\[ \eta_m(t,q)=\begin{cases}\prod_{i=1}^{m/2}(1-q^{2i}t) &m\ {\rm even}\\ \prod_{i=0}^{(m-1)/2}(1-q^{2i+1}t)&m\ {\rm odd}\end{cases}.\]

\begin{lemma}\label{functeq} Theorem~\ref{conj} is equivalent to $c_m(t,q)$ having the form
\[ c_m(t,q)={g_m(t,q)\over h_m(t)\eta_m(t,q)}\]
where $g_m(t,q)$ and $h_m(t)$ are polynomials with no common factor, have value 1 at $t=0$, and obey
\[   {q g_m(t,q)\eta_m(t,q^{-1})-q^{-1}g_m(t,q^{-1})\eta_m(t,q)\over q-q^{-1}}=h_m(t)\begin{cases}{1\over (1-t)}&m\ {\rm even}\\ 1 & {\rm else}\end{cases}.\] Moreover, $c_m(t,q)$ has rational $t,q$-degrees $-(m+1),-2$ respectively for $m\ge 2$. 
\end{lemma} 
\proof The content of Theorem~\ref{conj} is that $\zeta_t(V_m)=1/(\eta_m(t,q)\eta_m(t,q^{-1}))$ if $m$ is odd, with an extra factor $1/(1-t)$ if $m$ is even. Suppose $c_m(t,q)$ factorises in the form shown. Rearranging Lemma~4.3 now gives the stated equation for  $g_m,h_m$ (and conversely  if this equation holds then $c_m(t,q)$ of the form given implies Theorem~\ref{conj}). Note that the right hand side is independent of $q$. Next, by looking at the top power of $q$ in the numerator and denominator, either $2+\deg(g_m)>\deg(\eta_m)$, in which case the first term in the numerator dominates and $\deg(g_m)=0$ (and $\deg(\eta_m)<2$ which is possible for $m=0,1$), or $2+\deg(g_m)<\deg(\eta_m)$ in which case the second term dominates and $\deg(\eta_m)=2$ and $\deg(g_m)<0$ (a contradiction), or $2+\deg(g_m)=\deg(\eta_m)$. This gives the $q$-degree of $c_m(t,q)$. 
To see the $t$-degree we see from the equation for $g_m,h_m$ that each term of the numerator on the left has degree $\deg_t(g_m)+\deg_t(\eta_m)$. If $m\ge 2$ then by our previous analysis of $q$-degree these terms have different powers of $q$ in the coefficient of the top power of $t$ (the first term has highest power $\le-1$ in $q$ and the 2nd term has lowest power $\ge 1$). Hence there can be no cancellation for generic $q$ and $\deg_t(g_m)+\deg_t(\eta_m)$ is the degree of the numerator. The rational $t$-degree of $c_m(t,q)$ is then as stated, namely, the $t$-degree of $\zeta_t(V_m)$.   It remains to show that the equation shown for $g_m,h_m$ can be solved in the form stated. Indeed, this is a series of linear equations with coefficients rational functions of $t$. Thus, looking at the numerator, start with the top  degree in $q$, which is the coefficient of $q^{\deg(g_m)+1}$. Only the top coefficient in $g_m$ or the constant term in $g_m$ can contribute at this degree, which fixes the former in terms of the latter so that the numerator has nothing in this degree. We then consider the next degree, and so forth down to the coefficient of $q^2$. This solves for $g_m$ up to normalisation as a linear system of equations.  By antisymmetry in the form of the expression under inversion of $q$ we know that there is nothing in degree 0 and that the negative degrees will be solved as well by the same equations. Having solved for $g_m$ with coefficients as a rational function of $t$ we remove any factors that depend on $t$ alone and normalise so that $g_m(0,q)=1$ (say), and we define $h_m(t)$ by the displayed equation. \endproof

In fact by expanding $(q-q^{-1})\zeta_{t,m}(q)$ as a Laurent series in $q$ and using the residue theorem one can extract the positive part as 
\[ c_m(t,q)=q^{-1}\sum_{s=0}^{\lfloor{m-1\over 2}\rfloor}\sum_{z^{m-2s}=t}  {(-1)^s z^{s^2+s-1}\over (1-z^2)^{m-1}(m-2s)(1-qz) [s]_{z^2}![m-s]_{z^2}!},\quad\forall m\ge 2\]
and verify that this has the general form in Lemma~4.6. Here  $[\ ]_{z^2}$ is the notation in Section~2 and the factorials are  products of these. Note that summing over roots of $t$ necessarily produces a rational function by Kummer theory.

\begin{corollary} With reference to the form of $c_m(t,q)$ in Lemma~\ref{functeq}, 
\begin{eqnarray*} g_5(t,q)&=&1 - (q + q^3) t + (q^2 + q^4 + q^6) t^2- q^7 t^3 + q^4 t^4 + 
 (q  + q^3  - q^5) t^5 - (1+ q^4) t^6 \\ && + (2 q  + q^3   + 
 q^5) t^7 - (q^2 + q^4 + 2 q^6) t^8 + (q^3  + q^7) t^9 + 
( q^2  - q^4  - q^6 )t^{10} \\ && - q^3 t^{11} + t^{12} - (q +
 q^3 + q^5) t^{13} + (q^4  + q^6) t^{14} - q^7 t^{15}\\
 h_5(t)&=&(1-t^4)(1-t^6)(1-t^8)\\
 g_6(t,q)&=&1 + (1 - q^2  - q^4) t -( q^2  - q^6  - q^8) t^2 - (1-  q^2 - 
 q^4 - q^6 - q^8+ q^{10}) t^3 \\ && - (1- 2 q^2 -  q^4 +
 q^{10}) t^4 - (1-2 q^2 - q^4 + q^6 + q^8) t^5 \\  && + (q^2 + 
 q^4 - q^6  - 2 q^8  + q^{10}) t^6 + (1 - q^6  - 
 2 q^8 + q^{10}) t^7\\ &&  + (1 - q^2 - q^4  - q^6  - q^8  + 
 q^{10} )t^8 - (q^2 + q^4 - q^8) t^9 + (q^6  + q^8  - 
 q^{10}) t^{10}  - q^{10} t^{11}\\
 h_6(t)&=&(1+t)(1-t^2) (1 - t^3) (1 - t^4) (1 - t^5)
  \end{eqnarray*}
 \end{corollary}

We have independently verified all formulae for $c_m(t,q)$ to $m\le 6$ and to $t^9$ using the package LiE for the decomposition of $S^j(V_m)$. We note that the algorithm used by LiE (which does not lead to closed formulae) is based on Adams operations $\psi^j$. We recall that formally extending the direct sum and tensor products of representations
\[ S_t(V)= \sum_{j=0}^\infty t^j S^j(V)= e^{\sum_{j=1}^\infty {t^j\over j}\psi^j(V)}\]
which suggests an interpretation of the classical Hasse-Weil form of the $\zeta$-function with $\psi^j(V)$ providing $X(\F_{q^j})$ in a suitable category. 

Finally, the above results are to $\zeta_t(V)$ for $V$ irreducible as an object in our braided category. But because tensor products and direct sums under $q$-deformaton have the same structure up to isomorphism, property (\ref{lring}) remains true. Hence $\zeta_t$ of a direct sum module is the product of the $\zeta_t$ of each part. This means that we are in position to compute the braided $\zeta_t$ of the $q$-sphere, formally at least,  using Theorem~\ref{conj} under the assumption of generic $q$.  Namely we use $C_q[S^2]=\oplus_{s=0}^\infty V_{2s}$ as explained in Section~3, then formally
\[ \zeta_t(C_q[S^2])={1\over 1-t}\prod_{s=1}^\infty \zeta_t(V_{2s})=({1\over 1-t})^\infty\prod_{s=1}^\infty \prod_{i=1}^s{1\over (1-q^{2i}t)(1-q^{-2i}t)}.\]
Clearly we could remove the unquantized classical factor $({1\over 1-t})^\infty$ corresponding to the weight zero eigenspace of each component $V_{2s}$. It should be remembered, however, that this is a formal expression. For small powers of $t$ we can compute the coefficients directly for additional insight. 

\begin{proposition} For generic $q$, and provided summations over irreducibles are taken in a certain order and summands combined, 
\begin{eqnarray*}\zeta_t(C_q[S^2])&=&1- 2{ t \over(q-q^{-1})^2}+{4\over (2)_q^2}\left({t\over (q-q^{-1})^2}\right)^2+{2 ((4)_q^2-4)\over (2)_q^2(3)_q^2}\left({t\over(q-q^{-1})^2}\right)^3 +O(t^4)\end{eqnarray*}
\end{proposition}
\proof The $t^1$ coefficient is $\und\dim'(C_q[S^2])$ as in Section~3. For the $t^2$ coefficient  we have classically
\[ S^2(C(S^2))=S^2(\oplus_{s=0}^{\infty}V_{2s})=(\oplus_{s=0}^\infty S^2(V_{2s}))\oplus(\oplus_{s<s'}(V_{2s}\tens V_{2s'}))\]
and $S^2(V_{2s})=S^{2s}(V_2)=\oplus_{i=0}^sV_{4i}$ as noted above. Taking the same decompositions in the $q$-deformed case and using multiplicativity of $\und\dim'$ we compute
\begin{eqnarray*}  \und\dim'(S^2(C_q[S^2]))&=&\sum_{s=0}^\infty\sum_{i=0}^s(4i+1)_q+\sum_{0\le s<s'\le \infty}(2s+1)_q(2s'
+1)_q\\
&&={4\over (1-q^2)(1-q^{-2})(1-q^4)(1-q^{-4})}.\end{eqnarray*}
Note that each sum here is infinite but when the summands over $s$ are combined one obtains the result stated. Similarly, for $t^3$ we have
\[ S^3(C(S^2))=\left(\oplus_{s=0}^\infty S^3(V_{2s})\right)\oplus\left(\oplus_{s=0}^\infty (S^2(V_{2s})\tens(\oplus_{s'\ne s} V_{2s'}))\right)\oplus\left(\oplus_{s''< s< s'}(V_{2s''}\tens V_{2s}\tens V_{2s'})\right)\]
as $S^3(\oplus_{s=0}^\infty V_{2s})$. We decompose $S^3(V_{2s})=S^{2s}(V_3)$ either by the decomposition provided by  Corollary~\ref{c3c4} or, more easily, we compute the braided dimension directly from $\zeta_{\pm 1}(V_3)$. This works for all odd $m$:
\[ \und\dim'(\oplus_{s=0}^\infty S^m(V_{2s}))=\und\dim'(\oplus_{s=0}^\infty S^{2s}(V_m))=({\rm even\ part\ }\zeta_t(V_m))_{t=1}={1\over 2}(\zeta_1(V_m)+\zeta_{-1}(V_m))\]
and in the present case, using Theorem~\ref{conj} for $m=3$, gives
\[ {q^{-4}+q^{-2}+4+q^2+q^4\over (1-q^2)(1-q^{-2})(1-q^6)(1-q^{-6})}.\]
For the remaining direct sums we decompose $S^2(V_{2s})$ as before and apply $\und\dim'$ to give 
\[ \sum_{s=0}^\infty(\sum_{i=0}^s(4i+1)_q)(\sum_{s'=0, s'\ne s}^\infty(2s'+1)_q)+ \sum_{s=1}^\infty \sum_{s''=0}^{s-1}\sum_{s'=s+1}^\infty (2s''+1)_q(2s+1)_q(2s'+1)_q.\]
The two sums do not separately converge but merging the summands to a single sum over $s$ (with no contribution from  $s=0$ in the second expression), the combined sum then converges. We add this to our result from $\zeta_{\pm 1}$, express the denominators in terms of $q$-integers and group the poles as a rescaling of $t$. It should be noted that many different answers can be obtained when combining divergent series to convergent ones; the above is just one such `minimal prescription' in the absence of a general framework.  \endproof

We see that at least to low degree the coefficients of $\zeta_t(C_q[S^2])$ are rational functions in $q$, but this does depend on taking sums over irreducibles in the right manner. It also appears that the $\zeta$-function may  be expressible as a `renormalised'  function $\bar\zeta(t/(q-q^{-1})^2)$ with $\bar\zeta$ now having a classical limit as $q\to 1$. This would be reminiscent of the process of renormalisation in quantum field theory, but again will depend on the  scheme used. If we could eliminate the $(4)_q^2$ in one of the terms then we would have, for example, the leading terms of a $q$-Bessel function. A systematic treatment of these issues would be a topic for further work.

 \section{Braided Hilbert series}
 
Going back to the finite set case in Section~2, an alternative approach to $A^{(j)}$ there is as the degree $j$ part of the symmetric algebra $S(A)$. In the strictly braided case this will normally be different from the braided zeta function but in the $q$-deformed case for generic $q$ it coincides and offers a different point of view.

Let $A$ be an object in an Abelian braided category. One can define a canonical `braided-symmetric algebra'
\begin{equation}\label{braidedsym}BS(A)=T A/\oplus_j \ker({\CS}_j),\quad {\CS}_j=\sum_{\sigma\in S_j}\Psi_{i_1}\cdots\Psi_{i_{l(\sigma)}}\end{equation}
where $\sigma=s_{i_1}\cdots s_{i_{l(\sigma)}}$ is a reduced expression in terms of simple (adjacent) reflections and $\Psi_i$ denotes $\Psi$ acting similarly in the $i,i+1$ tensor factors. Constructions of this type have been used (in an antisymmetric variant) for left-invariant differential forms on quantum groups\cite{Wor}. This $BS(A)$ is a `braided group'\cite{Ma:book} or Hopf algebra in the braided category, with respect to coaddition, and in the finite-dimensional (rigid) case it arises naturally as the coradical of the pairing between braided tensor Hopf algebras $TA$ and $TA^*$ \cite{Ma:dbos,Ma:difdou,Ma:perm}. In the ribbon case we define
\begin{equation}\label{braidedHilbert}H_t(A)=\sum_{j=0}^\infty t^j \und\dim'(BS(A)^j).\end{equation}
In the $q$-deformed setting with generic $q$ both $A^{(j)}$ and $BS(A)^j$ deform the same classical representation and hence have the same multiplicities as modules for the relevant quantum group. Hence their braided dimensions coincide.
 
We conclude with a complementary class of examples, now  in the ribbon braided category of $D(G)$-modules where $D(G)$ is the quantum double of a classical finite group $G$. The ribbon structure cancels with the braiding in such a way that $\und\dim'=\dim$ for any object (in the notation of \cite{Ma:book} the special elements $u,v,\nu$ all coincide). 
As objects we take $X\subset G$ a nontrivial conjugacy class $A=\C X$. This carries an action of the quantum double and the induced braiding is
\[ \Psi(x\tens y)=xyx^{-1}\tens x,\quad \forall x,y\in X.\]

It has been explained in \cite{Ma:perm} that the Fomin-Kirillov quadratic algebra\cite{FomKir} for the flag variety of type $A_{n-1}$ can be understood as a braided symmetric algebra $BS^{\rm quad}(A)$ (where we take only the quadratic relations) and it was proposed that $BS(A)$ generalises this construction to general Lie type. Here $G=S_n$ the symmetric group and $X=X_n$ is the conjugacy class of 2-cycles and  $|X_n|=\left({n\atop 2}\right)$.  At least for small $n$ the relations of $BS(A)$ are only the quadratic ones and the dimensions of each degree  are also known for small $n$ and vanish after some top degree. The pattern of top degrees is conjectured to be related to the structure of the Lusztig canonical bases for $A_{n-1}$\cite{Ma:perm}. The Hilbert series for small $n$ are quoted in \cite{FomKir} as
\[ H_t(\C X_2)=[2]_t\]
\[  H_t(\C X_3)=[2]_t^2[3]_t\]
\[ H_t(\C X_4)=[2]_t^2[3]_t^2[4]_t^2 \]
\[ H_t(\C X_5)=[4]_t^4[5]_t^2[6]_t^4.\]
Comparing these results with the classical zeta-function on finite sets in Section~2, we see that they resemble the $\zeta$-functions of collections of points with different amounts of `regularity' in the sense of contributing different factors  $[m]_t$ with $2\le m<\infty$. 

Clearly, $X$ above is a special case of the notion of a `braided set', i.e. a set equipped with a map $X\times X\to X\times X$ obeying the braid relations\cite{IvaMa}; the computation of $H_t$ for other braided sets is a direction for further work.

\end{document}